\documentclass[12pt]
{amsart}

\usepackage{amsmath}%
\usepackage{amsfonts}%
\usepackage{amssymb}%
\usepackage{amscd}
\usepackage{graphicx}
\usepackage{color}
\usepackage{bbm}
\usepackage[all,cmtip]{xy}
\usepackage{hyperref}

\renewcommand{\qedsymbol}{\begin{picture}(0,0)%
\includegraphics{smallHamSandwich.pstex}%
\end{picture}%
\setlength{\unitlength}{4144sp}%
\begingroup\makeatletter\ifx\SetFigFont\undefined%
\gdef\SetFigFont#1#2#3#4#5{%
  \reset@font\fontsize{#1}{#2pt}%
  \fontfamily{#3}\fontseries{#4}\fontshape{#5}%
  \selectfont}%
\fi\endgroup%
\begin{picture}(204,171)(4489,-1584)
\end{picture}%
}

\newtheorem{theorem}{Theorem}[section]

\newtheorem{lemma}[theorem]{Lemma}
\newtheorem{corollary}[theorem]{Corollary}

\newtheorem{remarks}[theorem]{Remarks}

\newcommand{\matousek}{Matou\v{s}ek}

\newcommand{\vrecica}{Vre\'{c}ica}
\newcommand{\zivaljevic}{\v{Z}ivaljevi\'{c}}

\newcommand{\RR}{\mathbbm{R}} %reals
\newcommand{\eps}{\varepsilon}

\begin{document}

\title[Note on Masspartitions by Hyperplanes]
  {A Note on Masspartitions by Hyperplanes}
\author{Benjamin Matschke}
\thanks{Supported by Studienstiftung des dt. Volkes and Deutsche Telekom Stiftung.}
\thanks{This small note is an extract of \cite[Chap. II]{Mat08}}
\date{Dec. 31, 2009}

\maketitle

%%%%%%%%%%%%%%%%%%%%%%%%%%%%%%%%%%%%%%%%%%%%%%%%%%%%%%%%%%%%%%%%%%%%%
\begin{abstract}
%%%%%%%%%%%%%%%%%%%%%%%%%%%%%%%%%%%%%%%%%%%%%%%%%%%%%%%%%%%%%%%%%%%%%
A triple of positive integers $(d,h,m)$ is \emph{admissible} if for any $m$ given masses in $\RR^d$ there exist $h$ hyperplanes that cut each of these masses into $2^h$ equal pieces. 
We present an elementary reduction which combined with results by Ramos (1996) yields all the admissible triples that were known up to now (with one exception) as well as new ones.
\end{abstract}

%%%%%%%%%%%%%%%%%%%%%%%%%%%%%%%%%%%%%%%%%%%%%%%%%%%%%%%%%%%%%%%%%%%%%
\section{Introduction}
%%%%%%%%%%%%%%%%%%%%%%%%%%%%%%%%%%%%%%%%%%%%%%%%%%%%%%%%%%%%%%%%%%%%%

A \emph{mass} in our sense is a finite measure on the Borel $\sigma$-algebra on~$\RR^d$ for which hyperplanes are zero-sets. It is \emph{cut by $h$ hyperplanes into~$2^h$ equal pieces} if all the~$2^h$ orthants given by the hyperplanes have the same measure. A triple of positive integers $(d,h,m)$ is \emph{admissible} if for any $m$ given masses in $\RR^d$ there exist $h$ hyperplanes that cut each of these masses into $2^h$ equal pieces. 

The famous Ham Sandwich Theorem states that $(d,h=1,m=d)$ is admissible, which is due to Banach \cite{Ste38}. The problem of finding other admissible tripes arose in the sixties. Gr\"unbaum \cite{Gru60} asked which triples of the form $(d,h=2^d,m=1)$ are admissible, and Hadwiger \cite{Had66} showed the admissibility of the triples $(d=3,h=1,m=3)$ and $(d=3,h=2,m=1)$. The general problem was posed by Ramos \cite{Ram96}. It admits interesting topological approaches, see e. g. \cite{Had66}, \cite{Ram96}, \cite{Ziv04}, \cite{MVZ06}, \cite[Chap. II]{Mat08}. It is seen as one of the prime model problems for equivariant algebraic topology; compare \cite[Sect. 3.1]{Mat03}.

If $(d,h,m)$ is admissible then so is $(d+1,h,m)$: Project the masses from $\RR^{d+1}$ to $\RR^d$, find cutting hyperplanes, and take their pre-image under the projection. Hence, let $\Delta(h,m)$ be the smallest dimension $d$ such that $(d,h,m)$ is admissible. The famous Ham Sandwich Theorem states $\Delta(h=1,m)=m$.

Here we present a new reduction, Lemma \ref{lemMassPartInductionOnMForDelta}. 

\section{Elementary Reductions}

The idea of the following Lemma \ref{lemMassPartInductionOnHForDelta} was already used by Hadwiger \cite{Had66} and by Ramos \cite{Ram96}.

\begin{lemma}
\label{lemMassPartInductionOnHForDelta}
$\Delta(h,m) \leq \Delta(h-1,2m)$, for all $h\geq 2$, $m\geq 1$. That is, if $(d,h-1,2m)$ is admissible then so is $(d,h,m)$.
\end{lemma}
\begin{proof} Suppose that $(d,h-1,2m)$ is admissible and we are given $m$ masses in $\RR^d$. We can first bisect them with one hyperplane using the Ham Sandwich Theorem, since $d\geq m$. The resulting $2m$ masses can then be cut into equal parts using $h-1$ further hyperplanes, by the definition of $\Delta(h-1,2m)$.
\end{proof}

The next reduction will give new admissible triples.

\begin{lemma}
\label{lemMassPartInductionOnMForDelta}
$\Delta(h,m) \leq \Delta(h,m+1)-1$, for all $h,m\geq 1$. That is, if $(d+1,h,m+1)$ is admissible then so is $(d,h,m)$.
\end{lemma}
\begin{proof}
Assume we are given $m$ masses in $\RR^{\Delta(h,m+1)-1}$, which is embedded in $\RR^{\Delta(h,m+1)}$ with last coordinate equal to zero. Add a ball with radius $\frac{1}{2}$ at the point $(0,\ldots,0,1)\in\RR^{\Delta(h,m+1)}$ and view it as another mass. Thicken the first $m$ masses by an $\eps>0$ into the direction of the last coordinate, such that they are now actually masses in $\RR^{\Delta(h,m+1)}$ (the thickened masses have in fact the property that hyperplanes are zero-sets). 

\begin{figure}[h]
\centering 
\input{massPartInductionOnMandD.pstex_t}
\\ An example for $m=1$ and $h=2$.
\end{figure}

We then find an equipartition of the $m+1$ masses by $h$ hyperplanes. All of the hyperplanes go through the point $(0,\ldots,0,1)$, therefore they intersect $\RR^{\Delta(h,m+1)-1}$ in hyperplanes of $\RR^{\Delta(h,m+1)-1}$. These yield an equipartition of the given $m$ masses up to a small error which depends on the chosen $\eps$. A limit argument finishes the proof (take a convergent subsequence).
\end{proof}

\section{Result}

E. Ramos \cite{Ram96} has shown among other things the following inequalities ($x\geq 0$): 
$\Delta(1,2^{x+1})\leq 2\cdot 2^x$, $\Delta(2,2^{x+1})\leq 3\cdot 2^x$, $\Delta(3,2^{x+1})\leq 5\cdot 2^x$, $\Delta(4,2^{x+1})\leq 9\cdot 2^x$, $\Delta(5,2^{x+1})\leq 15\cdot 2^x$. Note that the factors in front of $2^x$ are of the form $2^{h-1}+1$, except for $h=5$ it is $15=2^{h-1}-1$. 

His formula $\Delta(5,2^{x+1})\leq 15\cdot 2^x$ gives together with the above lemmas new admissible triples. 
\begin{corollary}
\label{corNewResults}
Let $h\geq 5$, $m\geq 2$ and write $m=2^q+r$ $(0<r\leq 2^q)$. Then
\[
\Delta(h,m)\leq 2^{h-5}(14\cdot 2^q+r).
\]
\end{corollary}
\begin{proof}
First verify the inequality for $h=5$ with Lemma \ref{lemMassPartInductionOnMForDelta}, then prove it for larger $h$ with Lemma \ref{lemMassPartInductionOnHForDelta}.
\end{proof}

We note that Ramos' mentioned results together with the above reduction lemmas give all known admissible triples, except for one, $\Delta(2,5)=8$, which is due to Mani-Levitska, \vrecica, and {\zivaljevic} \cite{MVZ06}. 

\section*{Acknowledgements}

I thank Julia Ruscher and G\"unter Ziegler for their support.

\vspace{0.5cm}
\noindent
\SMALL{Benjamin Matschke\\
Institut f\"ur Mathematik, MA 6--2\\
Technische Universit\"at Berlin, 10623 Berlin, Germany\\
benjaminmatschke@googlemail.com\\
}
\end{document}